\documentclass{amsart}

\usepackage{amssymb, amsthm, amsmath}
\usepackage{amsrefs}
\usepackage{graphics}
\usepackage[hidelinks]{hyperref}
\usepackage{xcolor}
\usepackage[all]{xy}
\usepackage[mathscr]{eucal}
\usepackage{enumerate}
\usepackage{thmtools}

\numberwithin{equation}{section}

\declaretheorem[name=Theorem, sibling=equation]{theorem}
\declaretheorem[name=Lemma, sibling=equation]{lemma}
\declaretheorem[name=Proposition, sibling=equation]{proposition}
\declaretheorem[name=Corollary, sibling=equation]{corollary}

\declaretheorem[name=Definition, style=definition, sibling=equation]{definition}

\declaretheorem[name=Remark, style=definition, sibling=equation]{remark}
\declaretheorem[name=Example, style=definition, sibling=equation]{example}

\DeclareMathOperator{\reg}{reg}
\DeclareMathOperator{\Mod}{-Mod}

\DeclareMathOperator{\HH}{H}
\DeclareMathOperator{\Ob}{Ob}

\DeclareMathOperator{\id}{id}

\newcommand{\C}{{\mathscr{C}}}
\newcommand{\A}{{\mathscr{A}}}
\newcommand{\B}{{\mathscr{B}}}
\newcommand{\N}{{\mathbb{N}}}
\newcommand{\hor}{{\mathrm{hor}}}
\newcommand{\ver}{{\mathrm{ver}}}

\newcommand{\FI}{{\mathrm{FI}}}
\newcommand{\ZZ}{{\mathbb{Z}_{\geqslant -1}}}

\title{Bounding regularity of $\FI^m$-modules}

\author{Wee Liang Gan}

\author{Khoa Ta}

\begin{document}

\begin{abstract}
   Let $\FI$ be a skeleton of the category of finite sets and injective maps, and $\FI^m$ the product of $m$ copies of $\FI$. We prove that if an $\FI^m$-module is generated in degree $\leqslant d$ and related in degree $\leqslant r$, then its regularity is bounded above by a function of $m$, $d$, and $r$. 
\end{abstract}

\maketitle

Let $\N$ be the set of nonnegative integers. For each $n\in \N$, we write $[n]$ for the set $\{1, 2, \ldots, n\}$; in particular, $[0]$ denotes the empty set $\emptyset$.
Let $\FI$ be the category whose objects are the sets $[n]$ for $n=0, 1, 2, \ldots$ and whose morphisms are the injective maps between the objects. Let $\FI^m$ be the product of $m$ copies of the category $\FI$. 

It is a well-known result of Church and Ellenberg \cite{ce} that if $V$ is an $\FI$-module generated in degree $\leqslant d$ and related in degree $\leqslant r$, then the (Castelnuovo-Mumford) regularity of $V$ is bounded above by $d+r-1$. For $m>1$, Gan and Li \cite{gl_product} proved that $\FI^m$-modules presented in finite degrees have finite regularity but their proof does not lead to a bound. The main goal of our present article is to prove that if $V$ is an $\FI^m$-module generated in degree $\leqslant d$ and related in degree $\leqslant r$, then its regularity is bounded above by a function of $m$, $d$ and $r$. 

The proof of our result proceeds by nested induction: the outer induction is over $m$ and the inner induction is over $d$. There are two main ingredients in our argument: 
\begin{itemize}
   \item We define a pair of spectral sequences converging to the homology of a module over a product of two categories. These two spectral sequences exist not only for $\FI^m$-modules but in a general setting.
   \item We use a generalization to $\FI^m$-modules of the long exact sequence of Church \cite{church}.
\end{itemize}

Let us mention some related works. $\FI^m$-modules were studied by Casto \cite{casto}, Gadish \cite{gadish}, Gan and Li \cite{gl_product}, Li and Ramos \cite{lira}, Li and Yu \cite{liyu}, and Zeng \cites{zeng1, zeng2}. 

This article is organized as follows. In Section \ref{sec:main result}, we state our main result. In Section \ref{sec:sp seq}, we define certain functors associated to modules over a product of two categories and we construct spectral sequences converging to the homology of these modules. In Section \ref{sec:les}, we recall certain functors associated to $\FI^m$-modules and we construct a long exact sequence following Church \cite{church}. In Section \ref{sec:main proof}, we give the proof of our main result.

We are grateful to the referees for providing us with many suggestions to improve the exposition and the results of this article.

\section{Main result} \label{sec:main result}

\subsection{}
Throughout this article, we fix a commutative ring $k$. For any category $\C$, a $\C$-module is a functor from $\C$ to the category of $k$-modules. A homomorphism from a $\C$-module $U$ to a $\C$-module $V$ is a natural transformation from the functor $U$ to the functor $V$.

Let $\C$ be a small category. We write $\Ob(\C)$ for the set of objects of $\C$. For any $X, Y\in \Ob(\C)$, we write $\C(X, Y)$ for the set of morphisms in $\C$ from $X$ to $Y$. Let $V$ be a $\C$-module. For any $X\in \Ob(\C)$, we write $V_X$ for $V(X)$. For any $f\in \C(X, Y)$, we write $f_*$ for the map $V(f): V_X\to V_Y$. 

Denote by $\C\Mod$ the category of $\C$-modules. Recall that $\C\Mod$ is an abelian category. For each $W\in \Ob(\C)$, we define a $\C$-module $M^{\C}(W)$ as follows:
\begin{itemize}
\item 
for each $X\in \Ob(\C)$, let $M^{\C}(W)_X$ be the free $k$-module with basis $\C(W, X)$; 

\item 
for each $g\in \C(X, Y)$, let 
\[ g_* : M^{\C}(W)_X \to M^{\C}(W)_Y \] 
be the $k$-linear map sending each $f\in \C(W, X)$ to the composition $gf\in \C(W, Y)$. 
\end{itemize}
It is easy to see that $M^{\C}(W)$ is a projective $\C$-module. We call $M^{\C}(W)$ a \emph{principal projective} $\C$-module. 

We say that a $\C$-module is a \emph{free} $\C$-module if it is a direct sum of principal projective $\C$-modules. Every $\C$-module is a homomorphic image of a free $\C$-module, thus the abelian category $\C\Mod$ has enough projectives. Every projective $\C$-module is isomorphic to a direct summand of a free $\C$-module.

\subsection{}
Assume that $\C$ is a skeletal small category. Define a relation $\preceq$ on $\Ob(\C)$ by $X\preceq Y$ if $\C(X,Y)\neq \emptyset$. We write $X\prec Y$ if $X\preceq Y$ but not $Y\preceq X$. We say that $\C$ is \emph{directed} if the relation $\preceq$ on $\Ob(\C)$ is a partial order.

\begin{example}
Recall that $\C$ is an \emph{EI-category} if every endomorphism in $\C$ is an isomorphism. If $\C$ is a skeletal EI-category, then it is directed; see \cite{luck}*{Section 9}. 
\end{example}

Suppose now that $\C$ is directed. Let $V$ be a $\C$-module. 
For any $X\in\Ob(\C)$, define a $k$-submodule $\widetilde{V}_X$ of $V_X$ by
\[ \widetilde{V}_X =  \sum_{W\prec X} \left( \sum_{f\in \C(W,X)} f_*(V_W) \right). \]
The assignment $X\mapsto \widetilde{V}_X$ defines a $\C$-submodule $\widetilde{V}$ of $V$. Let 
\[ \HH^\C_0 : \C\Mod \to \C\Mod \] 
be the functor defined by $\HH^\C_0 (V) = V/\widetilde{V}$. Then $\HH^\C_0$ is a right exact functor and we can define its left derived functors. For each integer $i\geqslant 1$, let 
\[ \HH^\C_i : \C\Mod\to\C\Mod \] 
be the $i$-th left derived functor of $\HH^\C_0$. We call $\HH^\C_i(V)$ the $i$-th \emph{$\C$-homology} of $V$. 

\subsection{}
Fix an integer $m\geqslant 1$. The category $\FI^m$ is a directed skeletal small category. 

Let $V$ be an $\FI^m$-module. For each $\mathbf{n}=(n_1, \ldots, n_m)\in \N^m$, we set 
\begin{gather*} 
   |\mathbf{n}| = n_1+\cdots+n_m \in \N, \\
   [\mathbf{n}] = ([n_1], \ldots, [n_m]) \in \Ob(\FI^m).
\end{gather*}
We write $V_\mathbf{n}$ for $V_{[\mathbf{n}]}$. We define the \emph{degree} of $V$ by 
\[ \deg V = \begin{cases} 
               \sup \{ |\mathbf{n}| \mid V_\mathbf{n} \neq 0 \} & \mbox{ if } V\neq 0,\\
               -1 & \mbox{ if } V=0.
            \end{cases} \]
We say that $V$ has \emph{finite degree} if $\deg V < \infty$. For any $i\in \N$, let 
\[ t_i(V) = \deg \HH^{\FI^m}_i(V). \] 

Let 
\[ \ZZ = \{-1\} \cup \N. \]
We have $\deg V \in \ZZ \cup \{\infty\}$. 

For any $d\in\ZZ$, we say that $V$ is \emph{generated in degree $\leqslant d$} if $t_0(V) \leqslant d$. Equivalently, $V$ is generated in degree $\leqslant d$ iff there exists an epimorphism $P \to V$ where
\[ P = \bigoplus_{j\in J} M^{\FI^m}([\mathbf{n}_j]) \]
for some indexing set $J$ and each $\mathbf{n}_j\in \N^m$ satisfies $|\mathbf{n}_j| \leqslant d$.

For any $d, r\in \ZZ$, we say that $V$ is \emph{generated in degree $\leqslant d$} and \emph{related in degree $\leqslant r$} if there exists a short exact sequence
\[ 0 \to U \to P \to V \to 0 \]
such that: 
\begin{itemize}
   \item $P$ is a free $\FI^m$-module generated in degree $\leqslant d$,
   \item $U$ is an $\FI^m$-module generated in degree $\leqslant r$.
\end{itemize}
Observe that we can choose $P$ with $t_0(P)=t_0(V)$ and in this case we have 
\begin{equation} \label{eq:relation degree}
   t_1(V) \leqslant t_0(U) \leqslant \max\{ t_0(V), t_1(V) \}. 
\end{equation}

We define the \emph{regularity} $\reg(V)$ of $V$ by 
\[ \reg(V) = \sup \{ t_i(V)-i \mid i\geqslant 0\}. \]

\begin{remark}
   In \cite{ce}, the regularity of an $\FI$-module $V$ is defined as $\sup \{ t_i(V)-i \mid i\geqslant 1\}$. For our proofs below, it is more convenient to take the supremum over the range $i\geqslant 0$ so that $t_i(V)\leqslant i + \reg(V)$ for all $i\in\N$.
\end{remark}

\begin{remark}
   Our notion of regularity is distinct from the notion of Castelnuovo-Mumford regularity of $\FI^m$-modules defined by Li and Ramos in \cite{lira}. We do not know of any relation between these two notions.
\end{remark}

\subsection{}
We shall define for each integer $m\geqslant 1$ a function
\[ \rho_m : \ZZ \times \ZZ  \to \ZZ. \]
The precise definition of $\rho_m$ is not needed for understanding the statement of our main result, Theorem \ref{main}, below. 

\begin{definition} \label{def:rho}
   Let $m,d,r\in\mathbb{Z}$ with $m\geqslant 1$ and $d, r\geqslant -1$.

   If $m=1$, then let
   \[ \rho_m(d,r) = \max\{d, d+r-1\}. \]

   If $m\geqslant 2$ and $d=-1$, then let 
   \[ \rho_m(d,r) = -1. \]
   
   If $m\geqslant 2$ and $d\geqslant 0$, then let 
   \[ \rho_m(d,r) = \max\{ \rho_{m-1}(\rho'_m(d,r), \rho''_m(d,r)), 1+\rho_m(d-1,r) \} \]
   where 
   \begin{align*}
      \rho'_m(d,r) &= \max\{ 2+\rho_m(d-1, r), r \}  ,\\
      \rho''_m(d,r) &=  \max\{ 3+\rho_m(d-1, r), 4+ \rho_1(d,r) + \rho_{m-1}(d,r) \}.
   \end{align*}
\end{definition}

\begin{theorem} \label{main}
   Let $m, d, r\in \mathbb{Z}$ with $m\geqslant 1$ and $d, r\geqslant -1$. Let $V$ be an $\FI^m$-module generated in degree $\leqslant d$ and related in degree $\leqslant r$. Then
   \[ \reg(V) \leqslant \rho_m(d,r). \]  
\end{theorem}

The proof of Theorem \ref{main} will be given in Section \ref{sec:main proof}. We do not expect the bound in Theorem \ref{main} to be sharp.

The following corollary gives a bound on the regularity of $V$ in terms of $t_0(V)$ and $t_1(V)$. 

\begin{corollary} \label{main_corollary}
   Let $m\in \mathbb{Z}$ with $m\geqslant 1$. Let $V$ be an $\FI^m$-module. Assume that $t_0(V)<\infty$ and $t_1(V)<\infty$. Then 
   \[ \reg(V) \leqslant \rho_m(t_0(V), t_1(V)). \]
\end{corollary}

The proof of Corollary \ref{main_corollary} is in Subsection \ref{subsec:last}.

\section{Spectral sequences} \label{sec:sp seq}

\subsection{}
In this section, we let $\C$ be a product category $\A\times \B$, where $\A$ and $\B$ are directed skeletal small categories; in particular, $\C$ is a directed skeletal small category. 

Let $V$ be a $\C$-module. For any $(X,Y)\in \Ob(\C)$, define $k$-submodules $V^\hor_{(X,Y)}$ and $V^\ver_{(X,Y)}$ of $V_{(X,Y)}$ by
\begin{gather*}
V^\hor_{(X,Y)} =  \sum_{W\prec X} \left( \sum_{f\in \A(W,X)} (f,\id_Y)_*\left(V_{(W,Y)}\right) \right),\\
V^\ver_{(X,Y)} =  \sum_{Z\prec Y} \left( \sum_{g\in \B(Z,Y)} (\id_X,g)_*\left(V_{(X,Z)}\right) \right).
\end{gather*}

\begin{lemma}
\begin{enumerate}[(i)]
\item
The assignment $(X,Y)\mapsto V^\hor_{(X,Y)}$ defines a $\C$-submodule $V^\hor$ of $V$.

\item 
The assignment $(X,Y)\mapsto V^\ver_{(X,Y)}$ defines a $\C$-submodule $V^\ver$ of $V$.

\item
One has: $\widetilde{V} = V^\hor + V^\ver$.
\end{enumerate}
\end{lemma}
\begin{proof}
(i) Let $f\in \A(W, X)$ and $(g,h)\in \C((X,Y), (X',Y'))$. Then
\[ (g,h)(f,\id_Y) = (gf, h) = (gf, \id_{Y'})(\id_W, h), \]
which implies 
\[ (g,h)_* \left( (f,\id_Y)_*\left(V_{(W,Y)}\right) \right) \subseteq (gf, \id_{Y'})_* \left( V_{(W, Y')} \right). \]
Moreover, $W\prec X$ implies $W\prec X'$. Therefore $(g,h)_* \left(V^\hor_{(X,Y)} \right) \subseteq V^\hor_{(X',Y')}$.

(ii) Similar to (i).

(iii) It is clear that $V^\hor + V^\ver \subseteq \widetilde{V}$.

Now suppose $(f, g)\in \C((W,Z), (X,Y))$ where $(W,Z)\prec (X,Y)$. Then $W\prec X$ or $Z\prec Y$. 

If $W\prec X$, then 
\[ (f, g)_* \left( V_{(W,Z)} \right) \subseteq (f, \id_Y)_* \left( V_{(W, Y)} \right) \subseteq V^\hor_{(X,Y)}. \]
If $Z\prec Y$, then
\[ (f, g)_* \left( V_{(W,Z)} \right) \subseteq (\id_X, g)_* \left( V_{(X, Z)} \right) \subseteq V^\ver_{(X,Y)}. \]
Hence $\widetilde{V} \subseteq V^\hor + V^\ver$.
\end{proof}

By the preceding lemma, we may define functors 
\begin{gather*} 
\HH^\hor_0 : \C\Mod \to \C\Mod, \qquad V\mapsto V/V^\hor;\\
\HH^\ver_0 : \C\Mod \to \C\Mod , \qquad V\mapsto V/V^\ver;
\end{gather*}
moreover, there are canonical isomorphisms
\begin{equation} \label{vhor and vver}
\HH^\ver_0 (\HH^\hor_0 (V)) \cong \HH^\C_0(V) \cong \HH^\hor_0 (\HH^\ver_0 (V)). 
\end{equation}
The functors $\HH^\hor_0$ and $\HH^\ver_0$ are right exact and we can define their left derived functors.
For each integer $i\geqslant 1$, let 
\[ \HH^\hor_i : \C\Mod\to\C\Mod \] 
be the $i$-th left derived functor of $\HH^\hor_0$, and let 
\[ \HH^\ver_i : \C\Mod\to\C\Mod \] 
be the $i$-th left derived functor of $\HH^\ver_0$. We call $\HH^\hor_i(V)$ the $i$-th \emph{horizontal homology} of $V$, and $\HH^\ver_i(V)$ the $i$-th \emph{vertical homology} of $V$.

\subsection{}
For each $X\in \Ob(\A)$, we have an inclusion functor $\B \to \C$ defined on objects by $Y\mapsto (X,Y)$ and on morphisms by $g\mapsto (\id_X, g)$. Thus we obtain a restriction functor 
\[ \C\Mod \to \B\Mod, \quad V\mapsto V_{(X,-)} \]
where
\[ (V_{(X,-)})_Y = V_{(X,Y)} \quad\mbox{ for all }Y\in \Ob(\B). \]

Similarly, if we fix $Y\in \Ob(\B)$, we have a restriction functor 
\[ \C\Mod \to \A\Mod, \quad V \mapsto V_{(-,Y)} \]
where
\[ (V_{(-,Y)})_X = V_{(X,Y)} \quad\mbox{ for all }X\in \Ob(\A). \]

\begin{lemma} \label{restriction of free modules}
   Let $(W,Z)\in \Ob(\C)$. 
   \begin{enumerate}[(i)]
      \item For each $X\in \Ob(\A)$, we have 
      \[  M^{\C}(W,Z)_{(X,-)} \cong \bigoplus_{f\in \A(W,X)} M^{\B}(Z). \]
      \item For each $Y\in \Ob(\B)$, we have 
      \[  M^{\C}(W,Z)_{(-,Y)} \cong \bigoplus_{g\in \B(Z,Y)} M^{\A}(W). \]
   \end{enumerate}
\end{lemma}

\begin{proof} 
   (i) We define a homomorphism 
   \[ \phi: M^{\C}(W,Z)_{(X,-)} \to \bigoplus_{f\in \A(W,X)} M^{\B}(Z), \]
   as follows: for each $Y\in \Ob(\B)$, let 
   \[ \phi_Y : M^{\C}(W,Z)_{(X,Y)} \to \bigoplus_{f\in \A(W,X)} M^{\B}(Z)_Y. \]
   be the $k$-linear map sending the element $(f,g)\in \C((W,Z), (X,Y))$ to the element $g\in \B(Z,Y)$ in the direct summand indexed by $f\in \A(W,X)$. It is easy to see that $\phi$ is an isomorphism.

   (ii) Similar to (i).
\end{proof}

\begin{lemma} \label{homology_and_restrict}
   Let $V$ be a $\C$-module. 
   \begin{enumerate}[(i)]
      \item Let $X\in \Ob(\A)$. For each $i\geqslant 0$, we have
      \[ (\HH^{\ver}_i(V))_{(X,-)} \cong \HH^{\B}_i(V_{(X,-)}). \] 
      \item Let $Y\in \Ob(\B)$. For each $i\geqslant 0$, we have
      \[ (\HH^{\hor}_i(V))_{(-,Y)} \cong \HH^{\A}_i(V_{(-,Y)}). \] 
   \end{enumerate}
\end{lemma}
\begin{proof}
   (i) The case $i=0$ is obvious. The case $i>0$ follows because the restriction functor is exact and, by Lemma \ref{restriction of free modules}, the restriction of a free $\C$-module is a free $\B$-module.
   
   (ii) Similar to (i).
\end{proof}

\subsection{}
The spectral sequences in the following theorem are special cases of the Grothendieck spectral sequence associated to the composition of two functors.

\begin{theorem} \label{grothendieck}
Let $V$ be a $\C$-module. Then there are two convergent first-quadrant spectral sequences:
\begin{gather*}
^{I}\!E^2_{pq} = \HH^\ver_p (\HH^\hor_q (V)) \Rightarrow \HH^{\C}_{p+q}(V),\\
^{II}\!E^2_{pq} = \HH^\hor_p (\HH^\ver_q (V)) \Rightarrow \HH^{\C}_{p+q}(V).
\end{gather*}
\end{theorem}

\begin{proof}
We claim that $\HH^\hor_0$ sends projective $\C$-modules to $\HH^\ver_0$-acyclic $\C$-modules. It suffices to verify the claim for principal projective $\C$-modules. 

Let $(W,Z)\in \Ob(\C)$ and let $X\in \Ob(\A)$. 

If $X\neq W$, then 
\[ (\HH^\hor_0 ( M^{\C}(W,Z)))_{(X,-)} = 0. \]
If $X=W$, then 
\begin{align*}
   (\HH^\hor_0 ( M^{\C}(W,Z)))_{(X,-)}  &\cong M^{\C}(W,Z)_{(W,-)} & \\
   &\cong \bigoplus_{f\in \A(W,W)} M^{\B}(Z) & \mbox{(by Lemma \ref{restriction of free modules})}.
\end{align*}
Thus 
\begin{align*}
   ( \HH^\ver_0( \HH^\hor_0 ( M^{\C}(W,Z)) ) )_{(X,-)} &\cong \HH^{\B}_0 ( (\HH^\hor_0 ( M^{\C}(W,Z)))_{(X, -)} ) & \mbox{(by Lemma \ref{homology_and_restrict})}\\
   &= 0 . &
\end{align*}
This proves the claim. 

Using \eqref{vhor and vver} and Grothendieck spectral sequence of the composition $\HH^\ver_0 \HH^\hor_0$, we obtain the first spectral sequence. Similarly for the second spectral sequence.
\end{proof}

\section{Long exact sequence} \label{sec:les}

\subsection{}
In this section, we give a generalization to $\FI^m$-modules of the long exact sequence which Church constructed for $\FI$-modules in \cite{church}. We start by recalling the functors $\mathbf{\Sigma}$, $\mathbf{K}$, $\mathbf{D}$ and their basic properties following \cite{liyu}. 

Let $i\in [m]$. Define $\mathbf{e}_i \in \N^m$ by 
\[ \mathbf{e}_i = (0, \ldots, 1, \ldots, 0), \]
where $1$ is in the $i$-th coordinate. There is a functor $\iota_i : \FI^m \to \FI^m$ defined on objects by 
\begin{align*} 
   \iota_i : \Ob(\FI^m) &\to \Ob(\FI^m), \\ 
   [\mathbf{n}] &\mapsto [\mathbf{n}+\mathbf{e}_i],
\end{align*}
and on morphisms by
\begin{align*}
    \iota_i : \FI^m([\mathbf{n}], [\mathbf{r}]) &\to \FI^m([\mathbf{n}+\mathbf{e}_i], [\mathbf{r}+\mathbf{e}_i]), \\
     (f_1, \ldots, f_m) &\mapsto (g_1, \ldots, g_m), 
\end{align*}
where
$g_j=f_j$ for all $j\neq i$ and
\[ g_i(t) = \left\{ \begin{array}{ll} 
1 & \mbox{ if } t=1, \\
f_i(t-1)+1 & \mbox{ if } t>1. \end{array} \right.\]
The $i$-th shift functor 
\[ \Sigma_i : \FI^m\Mod \to \FI^m\Mod \] 
is defined to be the pullback via $\iota_i$. Thus, for any $\FI^m$-module $V$ and $\mathbf{n}\in \N^m$, we have 
\[ (\Sigma_i V)_{\mathbf{n}} = V_{\mathbf{n}+\mathbf{e}_i}. \] 

For each $[\mathbf{n}]=([n_1], \ldots, [n_m])\in\Ob(\FI^m)$, let 
\[ \varpi_i : [\mathbf{n}] \to [\mathbf{n}+\mathbf{e}_i] \] 
be the morphism of $\FI^m$ whose $j$-th component is the identity map on $[n_j]$ for $j\neq i$, and whose $i$-th component is the map $[n_i] \to [n_i+1]$, $t \mapsto t+1$. 

Let $V$ be an $\FI^m$-module. We have a natural homomorphism 
\[ \varepsilon_i: V\to \Sigma_i V \] 
defined at each $[\mathbf{n}]\in\Ob(\FI^m)$ to be the map $V_{\mathbf{n}} \to V_{\mathbf{n} + \mathbf{e}_i}$ induced by the morphism $\varpi_i$. 
Let $K_i V$ and $D_i V$ be, respectively, the kernel and cokernel of $\varepsilon_i: V\to \Sigma_i V$. Thus we have the exact sequence
\[ \xymatrix{
    0 \ar[r] & K_i V \ar[r] & V \ar[r]^-{\varepsilon_i} & \Sigma_i V \ar[r] & D_i V \ar[r] & 0. } \]

It is easy to see that: for any 
\begin{align*}
   \mathbf{n}&=(n_1, \ldots, n_m)\in \N^m,\\ 
   \mathbf{r}&=(r_1, \ldots, r_m)\in \N^m,\\ 
   \mathbf{f}&=(f_1, \ldots, f_m)\in \FI^m([\mathbf{n}], [\mathbf{r}]),
\end{align*}
if $n_i < r_i$ and $v\in (K_iV)_{\mathbf{n}}$, then $\mathbf{f}_*(v)=0$.

\begin{lemma} \label{lem:restriction of K1V}
   Let $m, x\in \mathbb{Z}$ with $m\geqslant 2$ and $x\geqslant 0$. Let $V$ be an $\FI^m$-module. Then we have:
   \begin{align*}
      t_0((K_1V)_{(x,-)}) &\leqslant \max\{ -1, t_0(K_1V)-x \}, \\
      t_1((K_1V)_{(x,-)}) &\leqslant \max\{ -1, t_1(K_1V)-x \}. 
   \end{align*}
\end{lemma}

\begin{proof}
   Let us consider $K_1V$ as a module over $\FI\times \FI^{m-1}$. 
   
   First, observe that
   \begin{equation} \label{eq:verK1}
      \HH^\ver_0(K_1V) = \HH^{\FI^m}_0(K_1V). 
   \end{equation}
   Hence we have:
   \begin{align*}
      \HH^{\FI^{m-1}}_0((K_1V)_{(x,-)}) &\cong (\HH^\ver_0(K_1V))_{(x,-)} & \mbox{(by Lemma \ref{homology_and_restrict})}\\
      &=  (\HH^{\FI^m}_0(K_1V))_{(x,-)} & \mbox{(by \eqref{eq:verK1}).}
   \end{align*}
   Therefore 
   \begin{align*}
     t_0((K_1V)_{(x,-)}) &= \deg(  \HH^{\FI^{m-1}}_0((K_1V)_{(x,-)}) ) \\
     &= \deg( (\HH^{\FI^m}_0(K_1V))_{(x,-)} ) \\
     &\leqslant \max\{ -1, t_0(K_1V)-x\}.
   \end{align*}
   
   Next, observe that
   \begin{equation} \label{eq:horK1}
      \HH^\hor_0(K_1V) = K_1V.
   \end{equation}   
   Hence we have:
   \begin{align*}
      \HH^{\FI^{m-1}}_1((K_1V)_{(x,-)}) &\cong (\HH^\ver_1(K_1V))_{(x,-)} &  \mbox{(by Lemma \ref{homology_and_restrict})}\\
      &= (\HH^\ver_1(\HH^\hor_0(K_1V) ))_{(x,-)} & \mbox{(by \eqref{eq:horK1}).}
   \end{align*}
   Observe also that from the first spectral sequence in Theorem \ref{grothendieck} applied to $K_1V$, we have an epimorphism 
   \[ \HH^{\FI^m}_1(K_1V) \to \HH^\ver_1(\HH^\hor_0(K_1V) ). \]
   Thus we have an epimorphism
   \[ (\HH^{\FI^m}_1(K_1V))_{(x,-)} \to (\HH^\ver_1(\HH^\hor_0(K_1V) ))_{(x.-)}. \]
   Therefore 
   \begin{align*}
      t_1((K_1V)_{(x,-)}) &= \deg(  \HH^{\FI^{m-1}}_1((K_1V)_{(x,-)}) ) \\
      &= \deg(  (\HH^\ver_1(\HH^\hor_0(K_1V) ))_{(x,-)}  ) \\
      &\leqslant \deg( (\HH^{\FI^m}_1(K_1V))_{(x,-)} ) \\
      &\leqslant \max\{ -1, t_1(K_1V)-x \}.    
   \end{align*}
\end{proof}

Define the functors $\mathbf{\Sigma}, \mathbf{K}, \mathbf{D}$ on $\FI^m\Mod$ by 
\[ \mathbf{\Sigma}V = \bigoplus_{i=1}^m \Sigma_i V, \qquad
\mathbf{K}V = \bigoplus_{i=1}^m K_i V, \qquad 
\mathbf{D}V = \bigoplus_{i=1}^m D_i V. \]
We have the exact sequence
\[ \xymatrix{
   0 \ar[r] & \mathbf{K} V  \ar[r] & V^{\oplus m}  \ar[r] & \mathbf{\Sigma} V  \ar[r] & \mathbf{D} V  \ar[r] & 0. } \]
We note that the functors $\Sigma_i$ and $\mathbf{\Sigma}$ are exact, while the functors $D_i$ and $\mathbf{D}$ are right exact.

\begin{lemma} \label{degS}
   Let $V$ be an $\FI^m$-module. Then $\deg(V) \leqslant 1 + \deg(\mathbf{\Sigma} V)$. 
\end{lemma}

\begin{proof}
   If $\deg V\leqslant 0$, the lemma is obvious. 
   
   Assume $\deg V > 0$. Consider any $\mathbf{n}=(n_1,\ldots,n_m)\in \N^m$ such that $|\mathbf{n}|>0$ and $V_{\mathbf{n}}\neq 0$. Since $|\mathbf{n}|>0$, there exists $i\in[m]$ such that $n_i\geqslant 1$. For this $i$, we have 
   \[ (\Sigma_i V)_{\mathbf{n} - \mathbf{e}_i}=V_{\mathbf{n}}\neq 0, \] 
   thus $(\mathbf{\Sigma} V)_{\mathbf{n} - \mathbf{e}_i}\neq 0$. The lemma follows.
\end{proof}

\begin{lemma} \label{degD}
   Let $d,r\in\ZZ$ and let $V$ be an $\FI^m$-module generated in degree $\leqslant d$ and related in degree $\leqslant r$. If $d\geqslant 0$, then $\mathbf{D}V$ is an $\FI^m$-module generated in degree $\leqslant d-1$ and related in degree $\leqslant r$. 
\end{lemma}
\begin{proof}
   Let
   \[ 0 \to U \to P \to V \to 0\]
   be a short exact sequence where $P$ is a free $\FI^m$-module generated in degree $\leqslant d$ and $U$ is an $\FI^m$-module generated in degree $\leqslant r$. Since the functor $\mathbf{D}$ is right exact, we have an exact sequence 
   \[ \mathbf{D}U \to \mathbf{D}P \to \mathbf{D}V \to 0. \]
   The lemma now follows from \cite{liyu}*{Lemma 2.3}.
\end{proof}

In Lemma \ref{degD}, when $r\geqslant 0$, the proof shows that $\mathbf{D}V$ is in fact related in degree $\leqslant r-1$. For the sake of convenience, we stated the lemma in the slightly weaker form so that we do not need to distinguish between the case $r=-1$ and the case $r\geqslant 0$.

\subsection{}
Let $i\in [m]$. For each $p\in \N$, write $L_p D_i$ for the $p$-th left derived functor of the right exact functor $D_i$. 

\begin{lemma} \label{2-row}
   Let $V$ be an $\FI^m$-module. Then for each $i\in [m]$, we have:
   \begin{enumerate}[(i)]
      \item $L_1 D_i (V) \cong K_i(V)$.
      \item $L_p D_i (V) = 0$ for all $p\geqslant 2$.
   \end{enumerate}
\end{lemma}

\begin{proof}
   The $m=1$ case is proved in \cite{ce}*{Lemma 4.7}. The $m>1$ case is essentially the same so we give only a sketch of the argument.

   Let 
   \[ 0 \to U \to P \to V \to 0 \] 
   be a short exact sequence of $\FI^m$-modules where $P$ is free. Then by the long exact sequence of left derived functors and the fact that $L_p D_i(P)=0$ for all $p\geqslant 1$, we see that: 
   \begin{itemize}
      \item $L_1 D_i(V)$ is the kernel of the morphism $D_i(U) \to D_i(P)$.
      \item $L_p D_i(V) \cong L_{p-1} D_i(U)$ for all $p\geqslant 2$.
   \end{itemize}
   
   On the other hand, we have the following commuting diagram with exact rows:
   \[ \xymatrix{ 
      0 \ar[r] & U \ar[r] \ar[d] & P \ar[r] \ar[d] & V \ar[r] \ar[d] & 0 \\
   0 \ar[r] & \Sigma_i U \ar[r] & \Sigma_i P \ar[r] & \Sigma_i V \ar[r] & 0 
   } \] 
   By the snake lemma and the fact that $K_i(P)=0$ (see \cite{liyu}*{Lemma 2.3}), we see that $K_iV$ is the kernel of the morphism $D_i(U)\to D_i(P)$. Hence $L_1 D_i(V) \cong K_i(V)$. This proves that (i) holds for any $\FI^m$-module $V$.

   We deduce that $L_1 D_i(U) \cong K_i(U) \subset K_i(P) = 0$, thus $L_1 D_i(U) = 0$. Since $L_2 D_i(V) \cong L_1 D_i(U)$, it follows that $L_2 D_i(V) = 0$. This holds for any $\FI^m$-module $V$, hence $L_p D_i(V)=0$ for all $p\geqslant 2$. 
\end{proof}

\subsection{}

\begin{lemma} \label{church_lemma}
   Let $i\in [m]$.
\begin{enumerate}[(i)]
\item 
For any $\FI^m$-module $V$, one has: 
\[\Sigma_i \widetilde{V} = \widetilde{\Sigma_i V} + \varepsilon_i (V),\]
an equality of $\FI^m$-submodules of $\Sigma_i V$.

\item
There is an isomorphism of functors:
\[ \HH^{\FI^m}_0 \circ D_i \cong \Sigma_i \circ \HH^{\FI^m}_0. \]
\end{enumerate}
\end{lemma}

\begin{proof}
(i) Let $\mathbf{n}\in\Ob(\FI^m)$. We need to prove that
\[
   \Sigma_i \widetilde{V}_{\mathbf{n}} = \widetilde{\Sigma_i V}_{\mathbf{n}} + \varepsilon_i (V)_{\mathbf{n}}.
\]

First, observe that: 
\begin{itemize}
\item 
$\Sigma_i \widetilde{V}_{\mathbf{n}}=\widetilde{V}_{\mathbf{n}+\mathbf{e}_i}$, which is spanned by the set of all $f_*(V_{\mathbf{r}})$ where $\mathbf{r}\prec \mathbf{n}+\mathbf{e}_i$ and $f\in \FI^m(\mathbf{r}, \mathbf{n}+\mathbf{e}_i)$. 

\item 
$ \widetilde{\Sigma_i V}_{\mathbf{n}}$ is spanned by the set of all $(\iota_i(f'))_*(V_{\mathbf{r'}+\mathbf{e}_i})$ where $\mathbf{r'}\prec \mathbf{n}$ and $f'\in \FI^m(\mathbf{r'}, \mathbf{n})$.

\item
$\varepsilon_i (V)_{\mathbf{n}}=(\varpi_i)_* (V_{\mathbf{n}})$.
\end{itemize}
It follows that $\Sigma_i \widetilde{V}_{\mathbf{n}} \supseteq \widetilde{\Sigma_i V}_{\mathbf{n}} + \varepsilon_i (V)_{\mathbf{n}}$.

Next, suppose that $\mathbf{r}\prec \mathbf{n}+\mathbf{e}_i$ and $f\in \FI^m(\mathbf{r}, \mathbf{n}+\mathbf{e}_i)$. Write $f=(f_1, \ldots, f_m)$. It is easy to see that: 
\begin{itemize}
\item
if $1\in \mathrm{Im}(f_i)$, then $f_*(V_{\mathbf{r}})\subseteq \widetilde{\Sigma_i V}_{\mathbf{n}}$; 

\item 
if $1\notin \mathrm{Im}(f_i)$, then $f_*(V_{\mathbf{r}})\subseteq \varepsilon_i (V)_{\mathbf{n}}$. 
\end{itemize}
Hence $\Sigma_i \widetilde{V}_{\mathbf{n}} \subseteq \widetilde{\Sigma_i V}_{\mathbf{n}} + \varepsilon_i (V)_{\mathbf{n}}$. 

(ii) Let $V$ be an $\FI^m$-module. We have the exact sequence 
\[ V\to \Sigma_i V \to D_i V \to 0. \]
Applying the right exact functor $\HH^{\FI^m}_0$ gives the exact sequence 
\begin{equation} \label{eq:lem_1_a}
   \HH^{\FI^m}_0 (V) \to \HH^{\FI^m}_0 (\Sigma_i V) \to \HH^{\FI^m}_0 (D_i V) \to 0. 
\end{equation}

On the other hand, there is a short exact sequence 
\[ 0 \to \widetilde{V} \to V \to \HH^{\FI^m}_0(V) \to 0. \]
Applying the exact functor $\Sigma_i$ gives the short exact sequence
\[ 0 \to \Sigma_i\widetilde{V} \to \Sigma_iV \to \Sigma_i\HH^{\FI^m}_0(V) \to 0. \] 
Applying the right exact functor $\HH^{\FI^m}_0$ gives the exact sequence 
\[ \HH^{\FI^m}_0(\Sigma_i\widetilde{V}) \to \HH^{\FI^m}_0(\Sigma_iV) \to \HH^{\FI^m}_0( \Sigma_i\HH^{\FI^m}_0(V) )\to 0. \]
Since $\HH^{\FI^m}_0(\Sigma_i\HH^{\FI^m}_0(V))=\Sigma_i\HH^{\FI^m}_0(V)$, the above exact sequence is
\[ \HH^{\FI^m}_0(\Sigma_i\widetilde{V}) \to \HH^{\FI^m}_0(\Sigma_iV) \to \Sigma_i\HH^{\FI^m}_0(V) \to 0. \]
By (i), the image of $ \HH^{\FI^m}_0(\Sigma_i\widetilde{V})$ in $\HH^{\FI^m}_0(\Sigma_iV)$ is equal to the image of $\HH^{\FI^m}_0(\varepsilon_i(V))$ in $\HH^{\FI^m}_0(\Sigma_iV)$. Hence we have the exact sequence 
\begin{equation} \label{eq:lem_1_b}
   \HH^{\FI^m}_0 (V) \to \HH^{\FI^m}_0 (\Sigma_i V) \to \Sigma_i\HH^{\FI^m}_0(V) \to 0. 
\end{equation}
It follows from (\ref{eq:lem_1_a}) and (\ref{eq:lem_1_b}) that $ \HH^{\FI^m}_0 (D_i V) \cong \Sigma_i\HH^{\FI^m}_0(V)$.
\end{proof}

The following is a straightforward generalization of the long exact sequence of Church \cite{church}.

\begin{theorem} \label{church_les}
   Let $V$ be an $\FI^m$-module.
   \begin{enumerate}[(i)]
      \item Let $i\in [m]$. Then there is a long exact sequence 
      \begin{multline*}
         \ldots \to \HH^{\FI^m}_{p-1} (K_iV) \to \Sigma_i \HH^{\FI^m}_p (V) \to \HH^{\FI^m}_p (D_iV) \to 
         \HH^{\FI^m}_{p-2} (K_iV) \to \ldots \\
         \ldots \to \Sigma_i \HH^{\FI^m}_1 (V) \to \HH^{\FI^m}_1 (D_iV) \to 0.
      \end{multline*}
      \item There is a long exact sequence
      \begin{multline*}
         \ldots \to \HH^{\FI^m}_{p-1} (\mathbf{K}V) \to \mathbf{\Sigma} \HH^{\FI^m}_p (V) \to \HH^{\FI^m}_p (\mathbf{D}V) \to 
         \HH^{\FI^m}_{p-2} (\mathbf{K}V) \to \ldots \\
         \ldots \to \mathbf{\Sigma} \HH^{\FI^m}_1 (V) \to \HH^{\FI^m}_1 (\mathbf{D}V) \to 0.
      \end{multline*}
   \end{enumerate}
\end{theorem}

\begin{proof}
   (ii) follows from (i) by taking direct sum over all $i\in [m]$, so we only need to prove (i). 

   Since the functors $\HH^{\FI^m}_0$ and $D_i$ are right exact, and $D_i$ sends projective modules to projective modules (by \cite{liyu}*{Lemma 2.3}), we have a first-quadrant Grothendieck spectral sequence 
   \[ E^2_{pq} = \HH^{\FI^m}_p(L_q D_i(V)) \Rightarrow L_{p+q}(\HH^{\FI^m}_0\circ D_i)(V). \]
   By Lemma \ref{church_lemma} and the exactness of $\Sigma_i$, we have 
   \[ L_{p+q}(\HH^{\FI^m}_0\circ D_i)(V) \cong L_{p+q}(\Sigma_i\circ \HH^{\FI^m}_0)(V) \cong \Sigma_i \HH^{\FI^m}_{p+q}(V). \]
   Thus the spectral sequence above converges to $\Sigma_i\HH^{\FI^m}_*(V)$. 

   Now by Lemma \ref{2-row}, we have: 
   \[ E^2_{pq} = \begin{cases}
      \HH^{\FI^m}_p(D_i(V)) & \mbox{ if }q=0,\\
      \HH^{\FI^m}_p(K_i(V)) & \mbox{ if }q=1,\\
      0 & \mbox{ if }q\geqslant 2. 
   \end{cases}\]
   Hence the long exact sequence in (i) is the long exact sequence associated to a two-row spectral sequence (see for example \cite{weibel}*{Exercise 5.2.2}).
\end{proof}

The following corollary is well-known when $m=1$; see for example \cite{gan}*{Lemma 7} or \cite{ramos}*{Corollary 3.13}.

\begin{corollary} \label{cor:finite degree case}
   Let $V$ be an $\FI^m$-module. Assume that $\deg(V) < \infty$. Then 
   \[ \reg(V) \leqslant \deg(V). \] 
\end{corollary}

\begin{proof}
   For any integers $d\geqslant -1$ and $i\geqslant 0$, denote by $\mathfrak{S}(d, i)$ the following statement:
   \begin{equation*}
      \parbox{.85\textwidth}{\it If $V$ is an $\FI^m$-module such that $\deg(V)\leqslant d$, then $t_i(V)\leqslant i+d$.}
   \end{equation*}
   We shall prove that $\mathfrak{S}(d,i)$ is true by nested induction. The outer induction is over $d$ and the inner induction is over $i$. 
   
   It is easy to see that $\mathfrak{S}(-1, i)$ is true for all $i\geqslant 0$, and $\mathfrak{S}(d, 0)$ is true for all  $d\geqslant -1$.

   Now fix $d\geqslant 0$ and $i\geqslant 1$. Assume that: 
   \begin{itemize}
      \item $\mathfrak{S}(d-1, j)$ is true for all $j\geqslant 0$;
      \item $\mathfrak{S}(d, i-1)$ is true.
   \end{itemize}
   To prove that $\mathfrak{S}(d,i)$ is true, let $V$ be an $\FI^m$-module such that $\deg(V)\leqslant d$. We need to show that $t_i(V) \leqslant i+d$. 
   
   We have:
   \begin{align*}
      t_i(V) &= \deg \HH^{\FI^m}_i(V) & \\
      &\leqslant 1+\deg \mathbf{\Sigma} \HH^{\FI^m}_i(V) & \mbox{(by Lemma \ref{degS})} \\
      &\leqslant \max\{1+t_{i-1} (\mathbf{K}V) , 1+t_i(\mathbf{D}V)\} & \mbox{(by Theorem \ref{church_les}).}
   \end{align*}
   Since $\mathbf{K}V$ is a submodule of $V^{\oplus m}$, we have $\deg(\mathbf{K}V)\leqslant d$. Using the assumption $\mathfrak{S}(d, i-1)$, we obtain:
   \[ t_{i-1}(\mathbf{K}V) \leqslant i+d-1. \]
   Since $\mathbf{D}V$ is a surjective image of $\mathbf{\Sigma}V$, we have $\deg(\mathbf{D}V)\leqslant d-1$. Using the assumption $\mathfrak{S}(d-1, i)$, we obtain:       
   \[ t_i(\mathbf{D}V) \leqslant i+d-1. \]
   It follows from the above that $t_i(V) \leqslant i+d$, as desired. 
\end{proof}

\section{Proof of main result} \label{sec:main proof}

\subsection{}

First, a straightforward appplication of the spectral sequences in Theorem \ref{grothendieck} yields the following result.

\begin{proposition} \label{prop:1}
   Let $m, \alpha, \beta\in \mathbb{Z}$ with $m\geqslant 2$ and $\alpha, \beta\geqslant -1$. Let $V$ be an $\FI^m$-module such that for all $(n_1, n_2, \ldots, n_m)\in \N^m$:
   \begin{align*}
      \reg\left(V_{(-, n_2, \ldots, n_m)}\right) &\leqslant \alpha,\\
      \reg\left(V_{(n_1, -)}\right) &\leqslant \beta.
   \end{align*} 
   Then 
   \[ t_i(V) \leqslant \max\{-1, 2i + \alpha + \beta\} \quad \mbox{ for all }i\in \N. \]
\end{proposition}

\begin{proof}
   Let $\mathbf{n}=(n_1, n_2, \ldots, n_m) \in \N^m$. 
   
   Applying Theorem \ref{grothendieck} to $V$ as a module over $\FI\times \FI^{m-1}$, we have the spectral sequence
   \[  ^{II}\!E^2_{pq} = \HH^\hor_p (\HH^\ver_q (V)) \Rightarrow \HH^{\FI^m}_{p+q}(V).  \]
   By Lemma \ref{homology_and_restrict}, we have   
   \begin{equation*} 
      (\HH_q^\ver(V))_{(n_1,-)} \cong \HH_q^{\FI^{m-1}}( V_{(n_1,-)} ).  
   \end{equation*} 
   Since 
   \[ \deg \HH_q^{\FI^{m-1}}( V_{(n_1,-)} ) \leqslant q + \beta. \]
   it follows that   
   \[ (\HH^\ver_q (V))_{\mathbf{n}} = 0 \quad \mbox{ if } \quad n_2+\cdots+n_m > q + \beta. \]
   Since $n_1$ is an arbitrary element of $\N$, we deduce that 
   \[ (\HH^\ver_q (V))_{(-, n_2, \ldots, n_m)} = 0 \quad \mbox{ if } \quad n_2+\cdots+n_m > q + \beta. \]
   By Lemma \ref{homology_and_restrict} again, we have
   \[ (\HH^\hor_p (\HH^\ver_q (V)))_{(-, n_2, \ldots, n_m)} \cong \HH_p^{\FI}( (\HH^\ver_q (V))_{(-, n_2, \ldots, n_m)} ), \]
   hence
   \[ (\HH^\hor_p (\HH^\ver_q (V)))_{\mathbf{n}} = 0  \quad\mbox{ if }\quad n_2+\cdots+n_m > q + \beta.  \]
   It follows from the spectral sequence that 
   \[ (\HH^{\FI^m}_i (V))_{\mathbf{n}} = 0  \quad\mbox{ if }\quad n_2+\cdots+n_m > i + \beta. \]

   Similarly, using the spectral sequence 
   \[  ^{I}\!E^2_{pq} = \HH^\ver_p (\HH^\hor_q (V)) \Rightarrow \HH^{\FI^m}_{p+q}(V), \]
   we deduce that
   \[ (\HH^{\FI^m}_i (V))_{\mathbf{n}} = 0  \quad\mbox{ if }\quad n_1 > i + \alpha \]
   Therefore 
   \[ (\HH^{\FI^m}_i (V))_{\mathbf{n}} = 0 \quad\mbox{ if }\quad |\mathbf{n}| > 2i + \alpha + \beta. \]
\end{proof}

The next result is essentially due to a referee of this article.

\begin{proposition} \label{prop:2}
   Let $m, \alpha, \gamma \in \mathbb{Z}$ with $m\geqslant 2$ and $\alpha, \gamma \geqslant -1$. Let $V$ be an $\FI^m$-module such that for all $x\in \N$:
   \begin{align*}
      x + \reg\left(V_{(x, -)}\right) &\leqslant \gamma \quad \mbox{ if } x\leqslant \alpha;\\
      V_{(x, -)} &= 0 \quad \mbox{ if } x > \alpha. 
   \end{align*}
   Then 
   \[ \reg(V) \leqslant \gamma. \]
\end{proposition}

\begin{proof}
   Applying Theorem \ref{grothendieck} to $V$ as a module over $\FI\times \FI^{m-1}$, we have the spectral sequence
   \[  ^{II}\!E^2_{pq} = \HH^\hor_p (\HH^\ver_q (V)) \Rightarrow \HH^{\FI^m}_{p+q}(V).  \]
   Let $\mathbf{n}=(n_1, n_2, \ldots, n_m) \in \N^m$ with 
   \[ n_1 + n_2 + \cdots + n_m > p+q+\gamma. \]
   It suffices to prove that 
   \[ (\HH^\hor_p (\HH^\ver_q (V)))_{\mathbf{n}} = 0. \]
   To this end, we first prove the following claim.
   
   \medskip 
   
   {\bf Claim \ref{prop:2}.1.} We have:
   \[ \deg\left( (\HH^\ver_q(V))_{(-, n_2, \ldots, n_m)} \right) < \max\{0, n_1-p\}. \]
   
   \begin{proof}[Proof of Claim \ref{prop:2}.1]
      Let $x\in \N$ with $x\geqslant n_1-p$. We need to show that 
      \[ (\HH^\ver_q(V))_{(x, n_2, \ldots, n_m)} = 0. \] 
      By Lemma \ref{homology_and_restrict}, we have
      \[ (\HH^\ver_q(V))_{(x, n_2, \ldots, n_m)} \cong (\HH^{\FI^{m-1}}_q(V_{(x, -)}))_{(n_2, \ldots, n_m)}. \]
      There are two cases: (1) $x\leqslant \alpha$, (2) $x>\alpha$.

      \medskip

      {\bf Case 1: $x\leqslant \alpha$.}

      In this case we have:
      \begin{align*}
         n_2 + \cdots + n_m &> p+q+\gamma-n_1  \\
         &\geqslant p+q+x+  \reg\left(V_{(x, -)}\right) -n_1  \\
         &\geqslant p+q+(n_1-p) + \left(t_q\left(V_{(x,-)}\right) - q\right)-n_1 \\
         &\geqslant t_q\left(V_{(x,-)}\right).
      \end{align*}
      This implies that 
      \[ (\HH^{\FI^{m-1}}_q(V_{(x, -)}))_{(n_2, \ldots, n_m)} = 0. \]

      \medskip

      {\bf Case 2: $x>\alpha$.} 

      In this case we have $V_{(x, -)}=0$, hence 
      \[ (\HH^{\FI^{m-1}}_q(V_{(x, -)}))_{(n_2, \ldots, n_m)} = 0. \]

      \medskip 

      This completes the proof of Claim \ref{prop:2}.1.
   \end{proof}

     Now by Lemma \ref{homology_and_restrict}, we have
      \[  (\HH^\hor_p (\HH^\ver_q (V)))_{\mathbf{n}} \cong (\HH^{\FI}_p( (\HH^\ver_q (V) )_{(-, n_2, \ldots, n_m)}  ))_{n_1}. \]
   
      If $n_1 \leqslant p$, then Claim \ref{prop:2}.1 implies that 
      \[ (\HH^\ver_q(V))_{(-, n_2, \ldots, n_m)} = 0, \] 
      so 
      \[ (\HH^{\FI}_p( (\HH^\ver_q (V) )_{(-, n_2, \ldots, n_m)}  ))_{n_1} = 0. \]

      If $n_1 > p$, then Claim \ref{prop:2}.1 implies that 
      \[ n_1 > p + \deg\left( (\HH^\ver_q(V))_{(-, n_2, \ldots, n_m)} \right). \]
      Hence by Corollary \ref{cor:finite degree case}, we have 
      \[ (\HH^{\FI}_p((\HH^\ver_q(V))_{(-, n_2, \ldots, n_m)}))_{n_1} = 0. \] 
\end{proof}

\subsection{}

We collect here some simple facts on the functions $\rho_m$, $\rho'_m$, $\rho''_m$ (see Definition \ref{def:rho}).

\begin{lemma} \label{lem:rho increasing in d}
   Let $m,d,r\in \mathbb{Z}$ with $m\geqslant 1$ and $d, r\geqslant -1$. Let $x\in \N$ with $x\leqslant d+1$. Then 
   \[ \rho_m(d,r) \geqslant x + \rho_m(d-x,r). \]
\end{lemma}

\begin{proof}
   This is trivial if $d=-1$ so assume that $d\geqslant 0$. 

   It is plain that for any integer $c\geqslant 0$, we have:
   \[ \rho_m(c,r) \geqslant 1 + \rho_m(c-1,r). \] 
   The lemma now follows from taking $c=d, d-1, \ldots, d-x+1$. 
\end{proof}

\begin{corollary} \label{lem:rho at least d}
   Let $m,d,r\in \mathbb{Z}$ with $m\geqslant 1$ and $d, r\geqslant -1$. Then $\rho_m(d,r) \geqslant d$. 
\end{corollary}

\begin{proof}
   Take $x=d+1$ in Lemma \ref{lem:rho increasing in d}.
\end{proof}

\begin{lemma} \label{lem:compare rho' rho''}
   Let $m,d,r\in \mathbb{Z}$ with $m\geqslant 2$, $d\geqslant 0$, $r\geqslant -1$. Then 
   \[    \rho''_m(d,r) > \rho'_m(d,r).   \]
\end{lemma}

\begin{proof}
   Obviously 
   \[ 3+ \rho_m(d-1, r) > 2+ \rho_m(d-1, r). \]
   We also have:
\begin{align*}
   4+ \rho_1(d,r) + \rho_{m-1}(d,r) &\geqslant 4 + (d+r-1) + \rho_{m-1}(d,r) \\
   &\geqslant 4 -1 +r-1 -1 \\
   & > r.
\end{align*}
\end{proof}

\subsection{}

We can now prove our main result.

\begin{proof}[Proof of Theorem \ref{main}]
   For any integers $m\geqslant 1$ and $d\geqslant -1$, denote by $\mathfrak{T}(m, d)$ the following statement:
   \begin{equation*}
      \parbox{.85\textwidth}{\it For any integer $r\geqslant -1$, if $V$ is an $\FI^m$-module which is generated in degree $\leqslant d$ and related in degree $\leqslant r$, then $\reg(V)\leqslant \rho_m(d,r)$.}
   \end{equation*}
   We shall prove that $\mathfrak{T}(m,d)$ is true by nested induction. The outer induction is over $m$ and the inner induction is over $d$. 

   By \cite{ce}*{Theorem A}, we know that $\mathfrak{T}(1,d)$ is true for all $d\geqslant -1$. It is easy to see that $\mathfrak{T}(m,-1)$ is true for all $m\geqslant 1$. 

   Now fix $m\geqslant 2$ and $d\geqslant 0$. Assume that:
   \begin{itemize}
      \item $\mathfrak{T}(m-1, c)$ is true for all $c\geqslant -1$;
      \item $\mathfrak{T}(m, d-1)$ is true.
   \end{itemize}
   To prove that $\mathfrak{T}(m,d)$ is true, fix an integer $r\geqslant -1$ and an $\FI^m$-module $V$ which is generated in degree $\leqslant d$ and related in degree $\leqslant r$. We want to show that 
   \[ \reg(V) \leqslant \rho_m(d,r). \] 
   We do this in several steps. 
 
   \medskip 

   {\bf Step 1. Bound $t_2(V)$.}

   Let $(n_1, n_2, \ldots, n_m)\in \N^m$. By Lemma \ref{restriction of free modules}, we know that $V_{(-, n_2, \ldots, n_m)}$ and $V_{(n_1, -)}$ are generated in degree $\leqslant d$ and related in degree $\leqslant r$. From the base case $\mathfrak{T}(1,d)$ and assumption $\mathfrak{T}(m-1,d)$, it follows that 
   \begin{align*}
      \reg\left( V_{(-,n_2,\ldots, n_m)} \right) &\leqslant \rho_1 (d,r),\\
      \reg\left( V_{(n_1,-)}\right) &\leqslant \rho_{m-1} (d,r).
   \end{align*}
   Hence by Proposition \ref{prop:1}, we have 
   \[ t_i(V) \leqslant \max\{-1, 2i + \rho_1(d,r) + \rho_{m-1}(d,r)\} \quad \mbox{ for all } i\in \N. \]
   In particular, 
   \begin{equation} \label{eq:t2bound}
      t_2(V) \leqslant 4 + \rho_1(d,r) + \rho_{m-1}(d,r).
   \end{equation}

   \medskip 

   {\bf Step 2. Bound $t_0(\mathbf{K}V)$ and $t_1(\mathbf{K}V)$.}

   By Theorem \ref{church_les}, we have a long exact sequence 
   \begin{multline*} 
      \cdots \to \HH^{\FI^m}_3(\mathbf{D}V) \to  \HH^{\FI^m}_1(\mathbf{K}V) \to \mathbf{\Sigma}\HH^{\FI^m}_2(V) \\ 
      \to \HH^{\FI^m}_2(\mathbf{D}V) \to  \HH^{\FI^m}_0(\mathbf{K}V) \to \mathbf{\Sigma}\HH^{\FI^m}_1(V) \to \cdots
   \end{multline*}
   Thus 
   \begin{align*}
      t_0(\mathbf{K}V) &\leqslant \max\{ t_2(\mathbf{D}V) , \deg \mathbf{\Sigma}\HH^{\FI^m}_1(V) \}, \\
      t_1(\mathbf{K}V) &\leqslant \max\{ t_3(\mathbf{D}V) , \deg \mathbf{\Sigma}\HH^{\FI^m}_2(V) \}.
   \end{align*}

   By Lemma \ref{degD} and assumption $\mathfrak{T}(m, d-1)$, we have 
   \[ t_2(\mathbf{D}V) \leqslant 2+\rho_m(d-1, r). \]
   We also have 
   \begin{align*}
      \deg \mathbf{\Sigma}\HH^{\FI^m}_1(V) &\leqslant \deg \HH^{\FI^m}_1(V) \\ 
      &= t_1(V) \\ 
      &\leqslant r.
   \end{align*}
   We deduce that 
   \begin{equation} \label{eq:bound t0}
      t_0(\mathbf{K}V) \leqslant \rho'_m(d,r). 
   \end{equation}

   By Lemma \ref{degD} and assumption $\mathfrak{T}(m, d-1)$ again, we have 
   \[ t_3(\mathbf{D}V) \leqslant 3+\rho_m(d-1, r). \]
   We also have 
   \begin{align*}
      \deg \mathbf{\Sigma}\HH^{\FI^m}_2(V) &\leqslant \deg \HH^{\FI^m}_2(V) & \\
      &= t_2(V) & \\
      &\leqslant 4+ \rho_1(d,r) + \rho_{m-1}(d,r) & \mbox{(by \eqref{eq:t2bound})}.
   \end{align*}
   We deduce that 
   \begin{equation} \label{eq:bound t1}
      t_1(\mathbf{K}V) \leqslant \rho''_m(d,r).
   \end{equation}

   \medskip 

   {\bf Step 3. Bound $\reg(\mathbf{K}V)$.}

   By \eqref{eq:bound t0} and \eqref{eq:bound t1}, we have: 
   \begin{align*}
      t_0(K_1 V) &\leqslant \rho'_m(d,r),\\
      t_1(K_1 V) &\leqslant \rho''_m(d,r).
   \end{align*} 
   Take any $x\in \N$. Then by Lemma \ref{lem:restriction of K1V}, we have: 
   \begin{equation}\label{two inequalities t0 t1}
   \begin{aligned} 
      t_0((K_1V)_{(x,-)}) &\leqslant \max\{ -1, \rho'_m(d,r)-x \}, \\
      t_1((K_1V)_{(x,-)}) &\leqslant \max\{ -1, \rho''_m(d,r)-x \}. 
   \end{aligned} \end{equation}
   We now consider the following two cases: (1) $x\leqslant \rho'_m(d,r)$, (2) $x>\rho'_m(d,r)$. 
   
   \medskip 

   {\bf Case 1: $x\leqslant \rho'_m(d,r)$.}

   In this case, from \eqref{two inequalities t0 t1}, we have: 
   \begin{align*}
      t_0((K_1V)_{(x,-)}) &\leqslant \rho'_m(d,r)-x, \\
      t_1((K_1V)_{(x,-)}) &\leqslant \rho''_m(d,r).
   \end{align*}
   Using \eqref{eq:relation degree} and Lemma \ref{lem:compare rho' rho''}, we deduce that $(K_1V)_{(x,-)}$ is generated in degree $\leqslant \rho'_m(d,r)-x$ and related in degree $\leqslant \rho''_m(d,r)$. Hence by assumption $\mathfrak{T}(m-1,c)$ where $c=\rho'_m(d,r)-x$, we have 
   \[ \reg((K_1V)_{(x,-)}) \leqslant \rho_{m-1}(\rho'_m(d,r)-x, \rho''_m(d,r)). \]
   Using Lemma \ref{lem:rho increasing in d}, it follows that
   \[ \reg((K_1V)_{(x,-)}) \leqslant \rho_{m-1}(\rho'_m(d,r), \rho''_m(d,r)) - x. \]
   Thus 
   \[ x + \reg((K_1V)_{(x,-)}) \leqslant \rho_{m-1}(\rho'_m(d,r), \rho''_m(d,r)). \]

   \medskip

   {\bf Case 2: $x> \rho'_m(d,r)$.}

   In this case, it follows from \eqref{two inequalities t0 t1} that $(K_1V)_{(x,-)}=0$.

   \medskip

   From the conclusions of the two cases above, we can apply Proposition \ref{prop:2} to $K_1V$ and deduce that 
   \[ \reg(K_1V) \leqslant \rho_{m-1}(\rho'_m(d,r), \rho''_m(d,r)). \]
   Similarly, for each $s\in [m]$, we have
   \[ \reg(K_sV) \leqslant \rho_{m-1}(\rho'_m(d,r), \rho''_m(d,r)). \]
   Hence,
   \begin{equation} \label{eq:bound reg of KV}
      \reg(\mathbf{K}V) \leqslant \rho_{m-1}(\rho'_m(d,r), \rho''_m(d,r)). 
   \end{equation}

   \medskip 

   {\bf Step 4. Bound $\reg(V)$.}

   Take any $i\in \N$. Recall from Theorem \ref{church_les} that we have a long exact sequence 
   \begin{equation} \label{les in pf}
      \cdots \to \HH^{\FI^m}_{i-1} (\mathbf{K}V) \to \mathbf{\Sigma} \HH^{\FI^m}_i (V) \to \HH^{\FI^m}_i (\mathbf{D}V) \to \cdots 
   \end{equation}
   We have
   \begin{align*}
      t_i(V) &= \deg \HH^{\FI^m}_i(V) & \\
      &\leqslant 1+\deg \mathbf{\Sigma} \HH^{\FI^m}_i(V) & \mbox{(by Lemma \ref{degS})} \\
      &\leqslant \max\{1+t_{i-1} (\mathbf{K}V) , 1+t_i(\mathbf{D}V)\} & \mbox{(by \eqref{les in pf}).}
   \end{align*}   
   By \eqref{eq:bound reg of KV}, 
   \[ 1+t_{i-1} (\mathbf{K}V) \leqslant i + \rho_{m-1}(\rho'_m(d,r), \rho''_m(d,r)).\]
   By Lemma \ref{degD} and assumption $\mathfrak{T}(m,d-1)$, we have
   \[ 1+t_i(\mathbf{D}V) \leqslant 1+i+\rho_m(d-1, r). \]
   It follows from above that
   \begin{align*} 
      t_i(V) &\leqslant \max\{ i + \rho_{m-1}(\rho'_m(d,r), \rho''_m(d,r)), 1+i+\rho_m(d-1, r)\} \\
      &= i + \rho_m(d,r).
   \end{align*}
   We conclude that $\reg(V) \leqslant \rho_m(d,r)$, as desired.
\end{proof}

\subsection{} \label{subsec:last}

We now prove Corollary \ref{main_corollary}.

\begin{proof}[Proof of Corollary \ref{main_corollary}]
   There are two cases: (1) $t_0(V)\leqslant t_1(V)$, (2) $t_0(V)>t_1(V)$.

   \medskip 

   {\bf Case 1: $t_0(V)\leqslant t_1(V)$.} 
   
   In this case, we know by \eqref{eq:relation degree} that $V$ is generated in degree $\leqslant t_0(V)$ and related in degree $\leqslant t_1(V)$. Hence by Theorem \ref{main}, we have 
   \[ \reg(V) \leqslant \rho_m(t_0(V), t_1(V)). \]

   \medskip

   {\bf Case 2: $t_0(V)>t_1(V)$.} 
   
   In this case, let 
   \[ A = \{ \mathbf{n} \in \N^m \mid |\mathbf{n}| \leqslant t_1(V) \}. \]  

   Let $U$ be the smallest $\FI^m$-submodule of $V$ such that $U_{\mathbf{n}} = V_{\mathbf{n}}$ for all $\mathbf{n}\in A$. It is easy to see that 
   \[ t_0(U)\leqslant t_1(V), \]
   so we also have 
   \begin{equation}  \label{eq:t0U less than t0V}
      t_0(U) < t_0(V). 
   \end{equation}

   Let $Q=V/U$. We have a short exact sequence 
   \[ 0 \to U \to V \to Q \to 0. \]
   Observe that
   \begin{gather*}
      t_0(Q) \leqslant t_0(V),\\
      t_1(Q) \leqslant \max\{ t_0(U), t_1(V) \} \leqslant t_1(V).
   \end{gather*}
   
   We have $Q_{\mathbf{n}} = 0$ for all $\mathbf{n}\in A$. Thus there exists an epimorphism $\phi: P\to Q$ where $P$ is a free $\FI^m$-module such that $P_{\mathbf{n}}=0$ for all $\mathbf{n}\in A$. 
   
   Let $W$ be the kernel of $\phi$. Then $W_{\mathbf{n}}=0$ for all $\mathbf{n}\in A$. We have a short exact sequence 
   \[ 0 \to W \to P \to Q \to 0. \] 
   Thus we have a monomorphism $\HH^{\FI^m}_1(Q) \to \HH^{\FI^m}_0(W)$. Hence:
   \begin{align*}
      & W_{\mathbf{n}}=0 \quad \mbox{ for all }\mathbf{n}\in A\\
      \Longrightarrow \quad & (\HH^{\FI^m}_0(W))_{\mathbf{n}}=0 \quad \mbox{ for all }\mathbf{n}\in A\\
      \Longrightarrow \quad & (\HH^{\FI^m}_1(Q))_{\mathbf{n}} = 0 \quad \mbox{ for all }\mathbf{n}\in A.
   \end{align*}  
    Since $t_1(Q)\leqslant t_1(V)$, we must have $\HH^{\FI^m}_1(Q)=0$. It follows from \cite{liyu}*{Theorem 1.3} that $\HH^{\FI^m}_i(Q)=0$ for all $i\geqslant 1$. 

    We deduce that 
    \[ \HH^{\FI^m}_i(U) \cong \HH^{\FI^m}_i(V) \quad  \mbox{ for all } i\geqslant 1. \] 
    In particular, 
    \[ t_i(U) = t_i(V) \quad \mbox{ for all } i\geqslant 1. \] 
    
   Since $t_0(U) \leqslant t_1(V)$, we have $t_0(U)\leqslant t_1(U)$. Thus by \eqref{eq:relation degree} we know that $U$ is generated in degree $\leqslant t_0(U)$ and related in degree $\leqslant t_1(U)$. Therefore:
   \begin{align*}
      \reg(U) &\leqslant \rho_m(t_0(U), t_1(U)) \quad \mbox{ (by Theorem \ref{main}) } \\
      &< \rho_m(t_0(V), t_1(U)) \quad \mbox{ (by \eqref{eq:t0U less than t0V} and Lemma \ref{lem:rho increasing in d}) }\\
      &= \rho_m(t_0(V), t_1(V)).
   \end{align*}
   Thus for all $i\geqslant 1$, we have:
   \begin{align*}
      t_i(V) - i &= t_i(U)-i \\
      &< \rho_m(t_0(V), t_1(V)).
   \end{align*}
   
   It remains to see that $t_0(V) \leqslant \rho_m(t_0(V), t_1(V))$, but this is immediate from Corollary \ref{lem:rho at least d}.
\end{proof}

\end{document}